\documentclass[12pt]{amsart}
\usepackage[active]{srcltx}
\usepackage[all
]{xy} \xyoption{poly}

\usepackage{xspace}
\usepackage{enumitem}
\usepackage{mathtools}
\usepackage[mathscr]{eucal}
\usepackage{amssymb}
\usepackage{a4wide}
\usepackage{ifthen}
\usepackage{amscd}
\usepackage[hyperindex,bookmarksnumbered,
plainpages,backref]{hyperref}

\usepackage{xcolor}

\newtheorem{thm}{Theorem}
\newtheorem{prop}{Proposition}
\newtheorem{lem}{Lemma}
\newtheorem{Rem}{Remark}

\makeindex

\begin{document}

{\fontsize{12pt}{17pt}\selectfont
\title{Jordan homomorphisms of triangular algebras over noncommutative algebras}

\author{Oksana Bezushchak}
\address{Oksana Bezushchak: Faculty of Mechanics and Mathematics, Taras Shevchenko National University of Kyiv, Volodymyrska, 60, Kyiv 01033, Ukraine}
\email{bezushchak@knu.ua}


\maketitle

\begin{abstract}
	D.~Benkovi\v{c} described Jordan homomorphisms of algebras of triangular matrices over 
	a commutative unital ring without additive $2$-torsion. 
	We extend this result to the case of noncommutative rings and remove the assumption of additive torsion.
	
Let $R$ be an associative unital algebra over a commutative unital ring $\Phi$.  
Consider the algebra $T_n(R)$ of triangular $n \times n$ matrices over $R$, and its subalgebra $T_n^0(R)$ consisting of matrices whose main diagonal entries lie in $\Phi$.  
We prove that for any Jordan homomorphism of $T_n(R)$, its restriction to $T_n^0(R)$ is standard.
\end{abstract}

\section{Introduction}

Let $A$ and $B$ be  associative (not necessarily commutative)  algebras over an associative commutative ring. An additive map $\varphi: A\to B$ is called a \textbf{Jordan homomorphism} if it  satisfies the following identities:

\begin{enumerate}
	\item[(i)]  $\varphi(x^2)= \varphi(x)^2$ for all $x\in A$;
	\item[(ii)] $\varphi(xyx)= \varphi(x) \varphi(y) \varphi(x)$ for all $x, y \in A$.
\end{enumerate}
If the algebra $B$ does not have additive $2$-torsion, then the identity (ii) follows from (i). For elements $x,y\in A,$ denote $x\circ y = xy+yx.$ The assumption (i) implies $$\varphi(x\circ y)= \varphi(x)\circ \varphi(y).$$

Let us recall that an additive map $\psi:A\to B$ is called an \textbf{antihomomorphism} if  it  satisfies the identity: $\psi(ab)=\psi(b)\psi(a)$ for any elements $a,b\in A.$

Every algebra homomorphism or antihomomorphism is a Jordan homomorphism, but not vice versa.

A deep and well-studied topic in ring theory and Jordan structure theory concerns the natural conditions under which Jordan homomorphisms can be represented as combinations of homomorphisms and antihomomorphisms of associative algebras.   L.K.~Hua~\cite{Hua} showed that a Jordan automorphism of a division ring is an automorphism or an antiautomorphism. I.~Herstein~\cite{Her} proved that a Jordan homomorphism onto a  prime ring without additive $2$-torsion is a homomorphism or an antihomomorphism. M.~Bre\v{s}ar~\cite{Bresar,Bresar2} extended it to the description of Jordan homomorphisms onto  semiprime rings. On the other hand, W.S.~Martindale~\cite{Mar} and N.~Jacobson~\cite{Jacobson1}(see also K.~McCrimmon~\cite{McCrimmon}) described all Jordan  homomorphisms of rings with at least three pairwise orthogonal full idempotents. See \cite{Bresar_Zelmanov} for further recent results on Jordan homomorphisms.

A bit more precisely: let  $A^{\text{op}}$ be the opposite algebra of an algebra $A$; that is, $A$ and  $A^{\text{op}}$ have the same additive structure, but multiplication in $A^{\text{op}}$  is defined by $a^{\text{op}} b^{\text{op}}=(ba)^{\text{op}}$ for all $a$, $b\in A$. Consider the direct sum $A\oplus A^{\text{op}}$. The map $S:A \to A\oplus A^{\text{op}},$ defined by $S(a) = a+ a^{\text{op}}$ for all $a\in A$, is a Jordan homomorphism. We denote by  $S\langle   A\oplus A^{\text{op}}\rangle$  the subalgebra of $A\oplus A^{\text{op}}$, generated by all elements  of the form $a+ a^{\text{op}}$, where $a\in A$.

A Jordan homomorphism $\varphi: A\to B$ is called \textbf{standard} if there exists 
a homomorphism  $\chi: S\langle A\oplus A^{\text{op}}\rangle \to B$ such that the diagram 
$$
\xymatrix{A\ar[rr]\ar[rrd]_{\varphi}& ^S&
	S\langle A\oplus A^{\text{op}}\rangle \ar[d]^{\chi}
	\\ &&B} \qquad \quad \xymatrix{a\ar[rr]\ar[rrd]& ^S&
	a +a^{\text{op}}\ar[d]^{\chi}
	\\ &&\varphi(a)}$$
	is commutative.

In this paper, we  explore Jordan homomorphisms of triangular algebras over noncommutative algebras.

Let $ \Phi$ be an associative commutative ring with $1$, 
and let $R$ be an associative unital $ \Phi $-algebra.   The algebra \( T_n(R) \), \( n \geq 2 \), consists of upper triangular \( n \times n \) matrices:  
\[
T_n(R) = 
\begin{pmatrix}
R & R & \cdots & R \\
					0 & R & \cdots & R \\
					\vdots & \vdots & \ddots & \vdots \\
					0 & 0 & \cdots & R
				\end{pmatrix}.
\]
Consider  the ideal \( S_n(R) \) of strictly upper triangular matrices:
\[
S_n(R) = 
\begin{pmatrix}
	0 & R & \cdots & R \\
	0 & 0 & \cdots & R \\
	\vdots & \vdots & \ddots & \vdots \\
	0 & 0 & \cdots & 0
\end{pmatrix},
\]
 the subalgebra \[ \mathrm{Diag}(\Phi) = 
\begin{pmatrix}
\Phi & 0 & \cdots & 0 \\
	0 &\Phi & \cdots & 0 \\
	\vdots & \vdots & \ddots & \vdots \\
	0 & 0 & \cdots & \Phi
\end{pmatrix},
\] consisting of diagonal matrices, and the subalgebra
\[
T_n^0(R) = \mathrm{Diag}(\Phi) \ltimes S_n(R) = \begin{pmatrix}
\Phi & R & \cdots & R \\
0 & \Phi & \cdots & R \\
\vdots & \vdots & \ddots & \vdots \\
0 & 0 & \cdots & \Phi
\end{pmatrix}.
\]

In 1998,  L.~Moln\'{a}r and P.~\v{S}emrl  \cite{Molnar_Semrl} proved that automorphisms and antiautomorphisms are the only Jordan automorphisms of triangular algebras over a field with at least three elements.  This result was subsequently extended to a  description of isomorphisms of  triangular rings over more general  coefficient rings by K.~Beidar, M.~Bre\v{s}ar, M.~Chebotar \cite{BBC}, Yao~Wang and Yu~Wang \cite{Yiqiu_Wang}, T.-L.~Wong \cite{Wong}, and C.-K.~Liu and W.-Y.~Tsai~\cite{Liu_Tsai}. 

D.~Benkovi\v{c}~\cite{Benkovic}  proved that if \( R =\Phi  \)  
and the ring \( \Phi  \) has no additive $2$-torsion,  
then every Jordan $\Phi $-linear homomorphism of the \( \Phi  \)-algebra \( T_n(\Phi ) \) is standard.

We extend this result to an arbitrary  associative (not necessarily commutative) unital  algebra~\( R \). 

\begin{thm}\label{main_th1}
Let \( R \) be an arbitrary associative  \( \Phi  \)-algebra with identity, and let  $n \geq 2$.   
Suppose that	$\varphi: T_n(R) \to B$ is a Jordan homomorphism. Then the restriction of $\varphi$  to $T_n^0 (R)$  is standard.  \end{thm}

Theorem~\ref{main_th1} establishes universality of the  Jordan homomorphism 
	\[
	T_n^0(R) \longrightarrow S\langle T_n(R) \oplus T_n(R)^{\text{op}} \rangle,
	\]
	but it leaves open the question about the 	structure of  $S\langle T_n(R) \oplus T_n(R)^{\mathrm{op}} \rangle.$	In the next theorem, we prove the decomposition
	\[
	S\langle T_n^0(R) \oplus T_n^0(R)^{\text{op}} \rangle 
	\cong
	\operatorname{Diag}(\Phi) \ltimes 
	(S_n(R) \oplus S_n(R)^{\text{op}}),
	\]
	which leads to a description of Jordan homomorphisms.

\begin{thm}\label{main_th2} Let ${\varphi: T_n(R) \to B}$  be a Jordan homomorphism. Then there  exist a unique homomorphism $\psi_1: T_n^0(R)  \to B$ and a unique antihomomorphism $\psi_2: T_n^0(R) \to B$ such that:
	\begin{itemize}
		\item[$(i)$] $\varphi(a) = \psi_1(a) = \psi_2(a)$ for all $a \in \emph{\text{Diag}}(\Phi)$,
		\item[$(ii)$] $ \psi_1(a)  \psi_2(b)= \psi_2(b)  \psi_1(a)=0$ for all $a, b \in S_n(R)$,
		\item[$(iii)$] $\varphi(a) =\psi_1(a) + \psi_2(a)$ for all $a \in S_n(R)$.
	\end{itemize} \end{thm}

In the special case $R = \Phi$, Theorem~\ref{main_th2}  reduces to D.~Benkovi\v{c}'s theorem in  \cite{Benkovic}, except that we impose no assumptions regarding  $2$-torsion.

D.~Benkovi\v{c} proved in \cite{Benkovic} that  the Jordan homomorphism $T_n(\Phi) \to  T_n(\Phi) \oplus T_n(\Phi)^{\text{op}}, $ $a \to a+ a^{\text{op}},$   cannot be represented as the sum of a homomorphism and an antihomomorphism with orthogonal images.This explains the difference between   Theorem~\ref{main_th2}  and the formulations in  \cite{Bresar,Bresar2,Her}.

In Sec.~\ref{Sec3}, we give an example showing that the restriction to $T_n^0(R)$ in Theorems~\ref{main_th1} and \ref{main_th2} is essential\footnote{The author is grateful to M.~Bre\v{s}ar for suggesting this example.}.

Finally, using Theorem~\ref{main_th2}, we describe bimodule derivations of the ring $T_n(R).$

Let $M$ be a bimodule over the algebra $T_n(R)$.  
An $\Phi$-linear mapping $d : T_n(R) \to M$ is called a \textbf{derivation} 
if 
\[
d(ab) = d(a)b + a d(b), \quad \text{for all } a,b \in T_n(R).
\]

An $\Phi$-linear mapping $d : T_n(R) \to M$ is called an \textbf{antiderivation} 
if 
\[
d(ab) = b d(a) + d(b)a, \quad \text{for all } a,b \in T_n(R).
\]

An $\Phi$-linear mapping $d : T_n(R) \to M$ is called a \textbf{Jordan derivation} 
if 
\[
d(a^2) = d(a)a + a d(a), \qquad 
d(aba) = d(a)ba + a d(b)a + ab d(a) , 
\quad \text{for all } a,b \in T_n(R).
\]

\begin{thm}\label{main_th3}
	Let  $M$ be a bimodule over $T_n(R)$ and let $d : T_n(R) \to M$ be a Jordan derivation.  
	Then there exist a derivation $d_1 : T_n^0(R) \to M$ and an antiderivation $d_2 : T_n^0(R) \to M$ 
	such that the restriction of $d$ to $T_n^0(R)$ is equal to $d_1 + d_2$, and 
	\[
	d_2(\mathrm{Diag}(\Phi)) = d_2([S_n(R), S_n(R)]) = (0).
	\]
\end{thm}

The proof of this theorem is derived from Theorem~\ref{main_th2}, literally following the same steps as in~\cite{Benkovic}.

\section{Universal Special Enveloping Algebras}

Let $T(A)$ denote the tensor algebra over the $\Phi$-module $A$. Let $J$ be the ideal of $T(A)$ generated by all elements of the forms:
\[
a \otimes a - a^2,\quad a \otimes b \otimes a - aba,\quad \text{for } a, b \in A.
\]
Let  $U(A) = T(A) / J.$ 

The map \( u: A \to U(A) \), \( u(a) = a + J \in U(A) \),  
is a Jordan homomorphism. This homomorphism is  \textbf{universal}, that is,  
for any Jordan homomorphism \( \varphi: A \to B \),  
there exists a homomorphism \( \chi: U(A) \to B \) such that the diagram
$$
\xymatrix{A\ar[rr]\ar[rrd]_{\varphi}& ^u&
	U(A)\ar[d]^{\chi}
	\\ &&B} $$
is commutative. The algebra \( U(A) \) is referred to as the \textbf{universal special enveloping algebra}  
of the Jordan algebra \( A^{(+)} \); see \cite{Jacobson,McCrimmon2,ZhSShSh}.

There  exists a homomorphism $
\pi: U(A) \to S\langle A\oplus A^{\text{op}}\rangle$ that makes the diagram 
$$
\xymatrix{a\ar[rr]\ar[rrd]_{S}& ^u&
	u(a)\ar[d]^{\pi}  
	\\ &&\qquad \qquad a+a^{\text{op}}, \qquad a\in A,} $$
commutative.

Denote $$A^0 = T_n^0 (R), \quad U^0(A)=\langle   u(A^0)\rangle \subseteq U(A),$$ and let $\pi^0: U^0 (A) \to S\langle   A^0\oplus (A^0)^{\text{op}}\rangle$ be the restriction of the homomorphism~$\pi.$

The following proposition is key to the proof of Theorem~\ref{main_th1}.

\begin{prop}\label{P1}
	The homomorphism \( \pi^0 \) is an isomorphism.
\end{prop}
The proof of this proposition requires a few steps.

Let $1_A$ be the identity element of the algebra $A$. Without loss of generality, we assume that the algebra $B$ is generated by the image $\varphi(A).$  Let $e=\varphi(1_A).$ From the definition of a Jordan homomorphism, it follows that $e$ is  idempotent and that $\varphi(A) = e\varphi(A) e.$ Thus, $e=1_B$, the  identity of the algebra $B$.

Consider the idempotents in the algebra $A$: \[ e_i = \operatorname{diag}(0, \ldots, 0, 1, 0, \ldots, 0) , \quad \text{for} \quad 1 \leq i \leq n .\]   The algebra \( A \) is graded by the free abelian group  
\( \Gamma = \oplus _{i=1}^{n} \mathbb{Z} \omega_i \) of rank \( n \). Consider the subset  $
\Delta = \{ \omega_i - \omega_j \mid 1 \leq i < j \leq n \}.$
Then  
\[
A = A_0 + \sum_{\alpha \in \Delta} A_\alpha, \quad \text{with} \quad A_{ \omega_i - \omega_j } = e_i A e_j.
\]
The elements \( \varphi(e_i) \),  for \( 1 \leq i \leq n \), are pairwise orthogonal idempotents, satisfying
\[
\sum_{i=1}^n \varphi(e_i) = 1_B.
\] The \( \Gamma \)-grading on the algebra \( A \) naturally extends to a \( \Gamma \)-grading on the algebra  
\[
U(A), \quad \text{where} \quad U(A) = \bigoplus_{\alpha \in \Gamma} U(A_\alpha).
\]

The following two lemmas are analogous to Proposition~1 in \cite{Bezushchak_Kashuba_Zelmanov}.

\begin{lem}\label{L1} Let \( \alpha, \beta \in \Delta \), with \( \alpha + \beta \not\in \Delta \cup (0)\). Then  
\[
u(A_\alpha) u(A_\beta) = (0).
\] \end{lem}

\begin{proof} Let $ x, y \in A $ and suppose $  x y = y x =0. $  
Then   $
u(x)\circ u(y) = u(x \circ y) =0 .$  Hence, $u(x) u(y)=-u(y) u(x).$
Futhermore, for any element $a\in A,$ we have $$ u(xya +ayx)=u(x) u(y) u(a) + u(a) u(y) u(x) = u(x) u(y) u(a) - u(a) u(x) u(y) =0. $$ Therefore, the element $u(x) u(y)$ commutes with the element $u(a).$ Since $u(A)$ generates the algebra $U(A)$, we conclude that $u(x) u(y)$ lies in the center of $U(A).$

For any element $e\in A$, let $\text{ad}(e)$ denote the adjoint operator $\text{ad}(e):x \to [e,x],$ for $x\in A.$  Since $\text{ad}(e)^2 (x)= [e,[e,x]]=e^2 x + x e^2 - 2 exe,$ it follows that  $$ \text{ad}(u(e))^2 (u(x)) =u(\text{ad}(e)^2(x)) .$$ 

Let $\alpha = \omega_i - \omega_j,$ $\beta =\omega_i - \omega_k,$ with $1 \leq i < j< k \leq n.$ Let $a\in A_{\alpha},$ $b\in A_{\beta}.$ Then we have: $$\text{ad}(e_j)^2 (a)= a, \quad \text{ad}(e_j)^2 (b)= 0.$$
Hence,
\[
\operatorname{ad}(u(e_j))^2\big(u(a)\big) = u(a), \quad \operatorname{ad}(u(e_j))^2\big(u(b)\big) = 0.
\]
This yields:
\[
	\operatorname{ad}(u(e_j))^2\big(u(a)u(b)\big) = \Big(\operatorname{ad}(u(e_j))^2\big(u(a)\big)\Big) u(b)+ \\
	 u(a)\Big(\operatorname{ad}(u(e_j))^2\big(u(b)\big)\Big) + \]
\[	  2\, [u(e_j), u(a)] \, [ u(e_j) , u(b) ]= u(a) u(b)+ 2\, [u(e_j), u(a)] \, [ u(e_j) , u(b) ] = 0.
\]

Similarly,
\[
	\operatorname{ad}(u(e_j))^4\big(u(a)u(b)\big)  = \Big(\operatorname{ad}(u(e_j))^4\big(u(a)\big)\Big) u(b)+4  
	\Big(\operatorname{ad}(u(e_j))^3\big(u(a)\big)\Big) \Big(\operatorname{ad}(u(e_j))\big(u(b)\big)\Big) + \]
	\[	  6\, 	\Big(\operatorname{ad}(u(e_j))^2\big(u(a)\big)\Big) \Big(\operatorname{ad}(u(e_j))^2\big(u(b)\big)\Big) + 4  \Big(\operatorname{ad}(u(e_j))\big(u(a)\big)\Big) \Big(\operatorname{ad}(u(e_j))^3\big(u(b)\big)\Big)+ \]
	\[	u(a) \, \Big(\operatorname{ad}(u(e_j))^4\big(u(b)\big)\Big) =
	u(a) u(b)+ 4\, [u(e_j), u(a)] \, [ u(e_j) , u(b) ] =0.
\] This implies that $u(a) u(b) = 0$.

The case $\alpha = \omega_i - \omega_k$, $\beta = \omega_j - \omega_k$, with $i < j < k$, is similar.

Let $\alpha = \omega_i - \omega_j$ and $\beta = \omega_p - \omega_q$, with $i < j$, $p < q$; where $i,j,p,q$ are pairwise distinct. Then we have
\[
\text{ad}(e_i)^2 (a) = a, \quad \text{ad}(e_i)^2 (b) = 0 \quad \text{for all} \quad a\in A_{\alpha} \quad \text{and} \quad b\in A_{\beta}. 
\] 
Arguing as above, we get:
\[
\text{ad}(u(e_i))^2 \left( u(a) u(b) \right) = u(a) u(b) + 2\, [u(e_i), u(a)] \, [ u(e_i) , u(b) ]  = 0
\]
and
\[
\text{ad}(u(e_i))^4 \left( u(a) u(b) \right) = u(a) u(b) + 4\, [u(e_i), u(a)] \, [ u(e_i) , u(b) ]  = 0,
\]
which implies $u(a) u(b) = 0$.

Choose arbitrary elements $a, b \in R$. We need to show that 
\[
u(e_{ij}(a)) \, u(e_{ij}(b)) = 0.
\]
We have 
\[
e_{ij}(a) = e_{ij}(1) \circ e_{ii}(a), 
\quad e_{ij}(b) = e_{ij}(1) \circ e_{jj}(b).
\]
Hence
\[
u(e_{ij}(a)) = u(e_{ij}(1)) \circ u(e_{ii}(a)), 
\quad u(e_{ij}(b)) = u(e_{ij}(1)) \circ u(e_{jj}(b)),
\]
\[	u(e_{ij}(a)) u(e_{ij}(b)) 
	= \]
	\[ u(e_{ij}(1)) u(e_{ii}(a)) u(e_{ij}(1)) u(e_{jj}(b)) + u(e_{ij}(1)) u(e_{ii}(a))  u(e_{jj}(b)) u(e_{ij}(1)) 
 +  \]
 \[ u(e_{ii}(a)) u(e_{ij}(1)) u(e_{ij}(1))  u(e_{jj}(b)) +  u(e_{ii}(a)) u(e_{ij}(1)) u(e_{jj}(b)) u(e_{ij}(1)).
\]
In the first, third, and fourth summands we obtain:
\[
u(e_{ij}(1)) u(e_{ii}(a)) u(e_{ij}(1)) = 
u\!\left( e_{ij}(1) e_{ii}(a) e_{ij}(1) \right) = 0,
\]
\[
u(e_{ij}(1)) u(e_{ij}(1)) = 
u\!\left( e_{ij}(1) e_{ij}(1) \right) = 0,
\]
\[
u(e_{ij}(1)) u(e_{jj}(b)) u(e_{ij}(1)) = 
u\!\left( e_{ij}(1) e_{jj}(b) e_{ij}(1) \right) = 0.
\]
Since the idempotents  $u(e_i)$ and $u(e_j)$ are orthogonal, and $u(e_{ii}(a))\in u(e_i)U u(e_i),$ $u(e_{jj}(b))\in u(e_j)U u(e_j),$ it follows that $u(e_{ii}(a)) u(e_{jj}(b))=0.$ This completes the proof of the lemma. \end{proof}

\begin{lem}\label{L2} $U^0(A) =u(A^0)^3.$ \end{lem}

\begin{proof} To prove this, we need  to show that for any elements $$a_1, a_2, a_3, a_4 \in \bigcup_{\alpha \in \Delta} A_\alpha \cup \{e_1, \ldots, e_n\},  $$ the product $u(a_1) u(a_2) u(a_3) u(a_4)$ 
lies in $u(A^0)^3$. Clearly, we can arrange the factors $u(a_p)$ in an arbitrary order. If two elements  $a_p$ among $a_1, a_2, a_3, a_4 $ lie in the set $\{e_1, \ldots, e_n\}$, say $a_1 = e_i$, $a_2 = e_j$, then
$u(e_i) u(e_j) = \delta_{ij} u(e_i),$  which proves the claim. Here, $\delta_{ij}$ denotes the Kronecker delta.

If three elements $a_p$ among $a_1, a_2, a_3, a_4$ lie in 	$\bigcup_{\alpha \in \Delta} A_\alpha$, say $a_1\in A_\alpha$, $a_2 \in A_\beta,$ and  $a_3 \in A_\gamma,$  then at least one of the sums $ \alpha +\beta,$ $\alpha+ \gamma,$ or $\beta+ \gamma$ does not lie in $\Delta$. By Lemma~\ref{L1}, this implies the claim and completes the proof of the lemma. \end{proof}

The algebra $U(A)$ is equipped with an involution $^*$ that  fixes all elements from $u(A)$. Clearly, $U^0(A)^* =U^0(A).$ By Lemma \ref{L2}, for any element $x \in U^0(A)$, we have $x + x^* \in u(A^0).$

As above, we consider the  homomorphism 
\[
\pi \colon U(A) \to S\langle   A\oplus A^{\text{op}}\rangle, \qquad  \pi(u(a)) = a +a^{\text{op}},
\]
for all  \( a \in A \), together with its restriction $
\pi^0 \colon U^0(A) \to S\langle   A^0\oplus (A^0)^{\text{op}}\rangle.$ Let \( I = \ker \pi^0 \). Both algebras $U^0(A)$ and $S\langle   A^0\oplus (A^0)^{\text{op}}\rangle$ are equipped with  involutions, and $\pi^0$  is a homomorphism of involutive algebras. Hence \( I^* = I \) and \( I \cap u(A^0) = (0) \).

\begin{proof}[Proof of Proposition \ref{P1}] To prove the proposition, we need to show that  $I = (0).$

By Lemma \ref{L2},  $U^0(A) = u(A^0)^3.$ For all elements \( a,b,c \in A^0 \), and $x=u(a) u(b) u(c)$, we have:
\[
x+ x^*=u(a) u(b) u(c)+ u(c) u(b) u(a) = u(abc + cba) \in u(A^0).
\]
Hence, for an arbitrary element \( x \in I \), we obtain:
\[
x+ x^* \in I \cap u(A^0) = (0), \quad \text{and} \quad x^* =- x .
\]

Choose an arbitrary element \( x \in A^0 \). Then \( x u(a) \in I \). Therefore,
\[
(x u(a))^* = u(a) x^*  = - u(a) x= - x u(a).
\]

Thus, we showed that the element $x$ commutes with the element $u(a)$. Since $u(A^0)$ generates the $\Phi$-algebra $U^0(A)$, it follows that the ideal $I$ lies in the center of $U^0(A)$.	Let $U_+$ be the ideal of $U^0(A)$ generated by
	\[
	\sum_{\alpha \in \Delta} u(A_\alpha). 
	\]
	Clearly, 
	\[
	U^0(A) = \sum_{i=1}^{n} \Phi \, u(e_i) + U_+.
	\]	From Lemma \ref{L2}, it follows that
	\[
	U_+ = \sum_{1 \leq i \not= j \leq n} u(e_i) \, U_+ \, u(e_j).
	\]	Hence, the center of the algebra $U^0(A)$ lies in 
	\[
	\sum_{i=1}^{n} \Phi \, u(e_i) \quad \text{and} \quad I \subseteq \sum_{i=1}^{n} \Phi \, u(e_i).
	\]
	It is clear that the homomorphism $\pi$ maps any nonzero element of $\sum_{i=1}^{n} \Phi \, u(e_i)$ to a nonzero element. We have thus shown that  $I = \{0\}$, which completes the proof of the proposition. \end{proof}

\section{Jordan homomorphisms}\label{Sec3}

Let us prove Theorem~\ref{main_th1}.
\begin{proof}
	Let $\varphi : A \to B$ be a Jordan homomorphism with $A = T_n(R)$, $n \geq 2$. 
	By the universality of the Jordan homomorphism $u$ and Proposition~\ref{P1}, 
	there exists a homomorphism
	\[
	\chi : S\langle   A\oplus A^{\text{op}}\rangle \to B
	\]
	such that the diagram
	
	$$
	\xymatrix{A\ar[rr]\ar[rrd]_{\varphi}& ^S&
		S\langle   A\oplus A^{\text{op}}\rangle \ar[d]^{\chi}
		\\ &&B} $$
	is commutative. Consequently, the restriction diagram
		$$
	\xymatrix{A^0\ar[rr]\ar[rrd]& &
		S\langle   A^0\oplus (A^0)^{\text{op}}\rangle \ar[d]
		\\ &&B} $$
	is  also commutative, and hence  the restriction of	the homomorphism $\varphi$ to $A^0$ is standard.	
\end{proof}

Now let us prove Theorem~\ref{main_th2}.
\begin{proof}	We first show that
	\begin{equation}\label{Eq1}
S\langle   A^0\oplus (A^0)^{\text{op}}\rangle  
	= \{\, a + a^{\text{op}}  \mid a \in \operatorname{Diag}(\Phi) \, \}
	+ S_n(R) + S_n(R)^{\text{op}} .
		\end{equation}	The inclusion from left to right is immediate.  
	It remains to prove that $S_n(R)$ and $S_n(R)^{\text{op}}$ belong to the subalgebra of $A^0\oplus (A^0)^{\text{op}}$
	generated by elements $a + a^{\text{op}}$, $a \in A^0$.
	
	For $1 \leq i < j \leq n$ and $r \in R$, let $e_{ij}(r)$ denote the matrix
	having $r$ in the $(i,j)$-entry and zeros elsewhere. Then
	\[
	(e_i + e_i^{\text{op}})(e_{ij}(r) + e_{ij}(r)^{\text{op}}) = e_{ij}(r).
	\] Hence 	\[ e_{ij}(r) \in S\langle   A^0\oplus (A^0)^{\text{op}}\rangle  , \qquad
	S_n(R) \subseteq S\langle   A^0\oplus (A^0)^{\text{op}}\rangle.
	\] Similarly, 	\[ 
	S_n(R)^{\text{op}} \subseteq S\langle   A^0\oplus (A^0)^{\text{op}}\rangle.
	\] 
	
	Let $\varphi : A \to B$ be a Jordan homomorphism.  
	By Theorems~\ref{main_th1}, there exists a homomorphism
	\[
	\chi : \{\, a + a^{\text{op}} \mid 
	a \in \operatorname{Diag}(\Phi) \,\}  + S_n(R) + S_n(R)^{\text{op}}
	\;\longrightarrow B
	\]
	such that $\varphi(a) = \chi(a + a^{\text{op}} )$ for all $a \in A^0$.	The subalgebra
	\begin{equation}\label{Eq2}
	\{\, a +a^{\text{op}}  \mid a \in \operatorname{Diag}(\Phi) \,\} + S_n(R)
		\end{equation}
	is isomorphic to $S_n(R)$, while the subalgebra
	\begin{equation}\label{Eq3}
	\{\, a + a^{\text{op}}  \mid a \in \operatorname{Diag}(\Phi) \,\} + S_n(R)^{\text{op}}
		\end{equation}
	is isomorphic to $S_n(R)^{\text{op}}$.  	Let $\psi_1$ be the restriction of $\chi$ to  subalgebra~(\ref{Eq2}),	and let $\psi_2$ be the restriction of $\chi$ to  subalgebra~(\ref{Eq3}). The homomorphism $\psi_2$ may be regarded as an antihomomorphism from $A^0$ to $B$.  Then 
\[
\varphi(a) =\psi_1(a+a^{\text{op}} ) =\psi_2(a+a^{\text{op}} ), \quad \text{for all }  a \in \mathrm{Diag}(\Phi);
\]
\[
\varphi(a) = \psi_1(a) +\psi_2(a), \quad \text{for all } a \in S_n(R);
\]
\[
\psi_1(a)\psi_2(b) =\psi_2(b)\psi_1(a) = 0, \quad \text{for all  } a, b \in S_n(R).
\]

It remains to prove the uniqueness of  a homomorphism $\psi_1$ and  an antihomomorphism $\psi_2$. The pair $(\psi_1, \psi_2)$ gives rise to a homomorphism $\psi: S\langle   A^0\oplus (A^0)^{\text{op}}\rangle \to B$ such that the corresponding diagram
	$$
\xymatrix{A^0\ar[rr]\ar[rrd]& &
	S\langle   A^0\oplus (A^0)^{\text{op}}\rangle \ar[d]^\psi
	\\ &&B} $$
is   commutative. Since the algebra $S\langle   A^0\oplus (A^0)^{\text{op}}\rangle $ is generated by elements $a+a^{\text{op}},$ $a\in A^0,$ and $\psi(a+ a^{\text{op}})=\varphi(a),$ the homomorphism $\psi$ is unique. This implies the uniqueness of $\psi_1'$ and $\psi_2'$, and completes the proof of Theorem~\ref{main_th2}. \end{proof}

\begin{Rem} {\rm  Let us show that the restriction to $T_n^0(R)$ is essential.  
Let $R$ be a noncommutative $\Phi$-algebra, and let 
\(\psi : R \to B\) be a Jordan homomorphism that is not the sum of a homomorphism and an antihomomorphism 
(see~\cite{Bezushchak_Kashuba_Zelmanov,Bresar0,Jacobson2}).  Let $$p : T_n(R) \to R, \quad p\left(\begin{pmatrix}
	a_{11}  & \cdots & a_{1n} \\
	\vdots & \ddots & \vdots \\
	0           & \cdots & a_{nn}
\end{pmatrix}\right) = a_{11}.$$
Then the composition  $\psi \circ p : T_n(R) \to B$ is a Jordan homomorphism which cannot be represented as the sum of a homomorphism and an antihomomorphism. }\end{Rem}

\section*{Acknowledgment}

The author thanks  M.~Bre\v{s}ar, I.~Kashuba, and E.~Zelmanov for valuable and helpful discussions.

The author acknowledges support from the MES of Ukraine through the grant for the perspective development of the scientific direction Mathematical Sciences and Natural Sciences at TSNUK.

\end{document}